\documentclass[reqno,11pt]{amsart}
\usepackage{amsfonts,xcolor}
\usepackage{hyperref}

\newtheorem{theorem}{Theorem}[section]

\newtheorem{proposition}[theorem]{Proposition}
\newtheorem{corollary}[theorem]{Corollary}

\author{Yongtao  Li}
\address{ Yongtao Li, College of Mathematics and Econometrics, Hunan University, 
Changsha, Hunan, 410082, P.R. China}
\email{\url{ytli0921@hnu.edu.cn} }

\author{Yang Huang}
\address{ Yang Huang, School of Mathematics and Statistics, Central South University, 
Changsha, Hunan, 410083, P.R. China}
\email{\url{FairyHuang@csu.edu.cn} }

\author{Lihua Feng}
\address{Lihua Feng, School of Mathematics and Statistics, Central South University, 
Changsha, Hunan, 410083, P.R. China}
\email{\url{fenglh@163.com} }

\author{Weijun Liu}
\address{Weijun Liu, School of Mathematics and Statistics, Central South University, 
Changsha, Hunan, 410083, P.R. China}
\email{\url{wjliu6210@126.com}}

\date{\today}

\begin{document}

\title[Applications of two maps]{Some applications of two completely copositive maps}

\begin{abstract}
A linear map $\Phi :\mathbb{M}_n \to \mathbb{M}_k$ 
is called completely copositive if the resulting matrix $[\Phi (A_{j,i})]_{i,j=1}^m$ 
is positive semidefinite  for any integer $m$ and 
 positive semidefinite matrix $[A_{i,j}]_{i,j=1}^m$. 
In this paper, 
we present some applications of the completely copositive maps 
$\Phi (X)=(\mathrm{tr} X)I+X$ and $\Psi (X)=(\mathrm{tr} X)I-X$. 
Some new extensions about traces inequalities of 
positive semidefinite $3\times 3$ block matrices are included. 
\end{abstract}
\maketitle

\section{Introduction}

The space of $m\times n$ complex matrices is denoted by $\mathbb{M}_{m\times n}$. 
If $m=n$, we use $\mathbb{M}_n$ instead of $\mathbb{M}_{n\times n}$ and 
if $n=1$, we use $\mathbb{C}^m$ instead of $\mathbb{M}_{m\times 1}$. 
The  identity matrix of $\mathbb{M}_n$ is denoted by $I_n$, or simply by $I$ 
if no confusion is possible. 
We use $\mathbb{M}_m(\mathbb{M}_n)$ for the set of $m\times m$ block matrices 
with each block in $\mathbb{M}_n$. 
Let $X\otimes Y$ denote the Kronecker product of $X,Y$, that is, 
if $X=[x_{ij}]\in \mathbb{M}_m$ and $Y\in \mathbb{M}_n$, then 
$X\otimes Y\in \mathbb{M}_m(\mathbb{M}_n)$ whose $(i,j)$ block is $x_{ij}Y$. 
By convention, if $X\in \mathbb{M}_n$ is positive semidefinite, we write $X\ge 0$. 
For two Hermitian matrices $A$ and $B$ of the same size, $A\ge B$ means $A-B\ge 0$.  

Now we introduce the definition of the partial transpose and partial traces. 
For any $A=[A_{i,j}]_{i,j=1}^m\in \mathbb{M}_m(\mathbb{M}_n)$, 
the {\it   usual transpose} of $A$ is defined to be 
\[ A^T=\begin{bmatrix} A_{1,1}^T & \cdots & A_{m,1}^T \\
 \vdots & \ddots & \vdots \\
A_{1,m}^T & \cdots & A_{m,m}^T \end{bmatrix}. \]
We define the {\it  partial tranpose} of $A$ by 
\[ A^{\tau}=\begin{bmatrix} A_{1,1} & \cdots & A_{m,1} \\
 \vdots & \ddots & \vdots \\
A_{1,m} & \cdots & A_{m,m} \end{bmatrix}. \]
It is clear that $A\ge 0$ does not necessarily imply $A^{\tau}\ge 0$. 
If both $A$ and $A^{\tau}$ are positive semidefinite, 
then $A$ is said to be the positive partial tranpose (or PPT for short).

Given $A=[A_{i,j}]_{i,j=1}^m\in \mathbb{M}_m(\mathbb{M}_n)$, 
we next introduce the definition of two partial traces $\mathrm{tr}_1A$ and $\mathrm{tr}_2A$. 
There are several equivalent ways to explain the partial traces, 
and we recommend the recent monographs \cite{Petz08} and 
\cite[pp.120--121]{Bhatia07} for a comprehensive survey of the subject. 
For notational convenience, we define two partial traces  of $A$ by 
\begin{align*} 
\mathrm{tr}_1A &=\sum\limits_{i=1}^m A_{i,i}, \\
\mathrm{tr}_2 A & =[\mathrm{tr} A_{i,j}]_{i,j=1}^m, 
\end{align*} 
where $\mathrm{tr} X$ stands for the usual trace of $X$. 

Recall that a linear map $\Phi: \mathbb{M}_n\to \mathbb{M}_k$ is called positive if it maps positive matrices 
to positive matrices. 
A linear map $\Phi: \mathbb{M}_n\to \mathbb{M}_k$  is said to be $m$-posotive if 
for $[A_{i,j}]_{i,j=1}^m\in \mathbb{M}_m(\mathbb{M}_n)$, 
\begin{equation}  \label{eq1}
[A_{i,j}]_{i,j=1}^m \ge 0 \Rightarrow [\Phi (A_{i,j})]_{i,j=1}^m\ge 0. 
\end{equation}
It is said to be completely positive 
if (\ref{eq1}) holds for any integer $m\ge 1$. 
It is well known that the trace map and the determinant map 
are both completely positive, see, e.g., \cite[p. 221, p. 237]{Zhang11}. 
On the other hand, a linear map  $\Phi $  is said to be $m$-coposotive if 
for $[A_{i,j}]_{i,j=1}^m\in \mathbb{M}_m(\mathbb{M}_n)$, 
\begin{equation}  \label{eq2}
[A_{i,j}]_{i,j=1}^m \ge 0 \Rightarrow [\Phi (A_{j,i})]_{i,j=1}^m\ge 0,  
\end{equation}
and $\Phi$ is said to be completely copositive 
if (\ref{eq2}) holds for any positive integer $m\ge 1$. 
Furthermore, 
$\Phi$ is called a completely PPT map if it is completely positive and completely copositive. 
A comprehensive survey of the standard results on completely positive maps can be found in 
\cite[Chapter 3]{Bhatia07} or \cite{Paulsen02}. 

This paper centers on the application of the following result (Theorem \ref{thm1}) 
due to Lin \cite[Theorem 1.1]{Lin14} and \cite[Proposition 2.1]{Lin16}.  
We provide an alternatively elegant  proof here. 
%Our proof is more elegant and succinct. 

\begin{theorem} \label{thm1}
(see \cite{Lin14,Lin16}) 
The maps $\Phi (X)=(\mathrm{tr}X)I+ X$ 
and $\Psi (X)= (\mathrm{tr}X)I-X$ are completely copositive. 
\end{theorem}

\begin{proof}
%[Alternative proof]
Here we use a standard method from Choi \cite{Choi75}. 
It is sufficient to show that for any positive integer $m$, 
\[ \left[\Phi (E_{j,i})\right]_{i,j=1}^m \ge 0 \text{~~and~~}   
\left[\Psi (E_{j,i}) \right]_{i,j=1}^m \ge 0, \]
where $E_{j,i}\in \mathbb{M}_n$ is the matrix with $1$ in the $(j,i)$-th entry and $0$ elsewhere. 
Since both $[\Phi (E_{j,i})]_{i,j=1}^m$ and $[\Psi (E_{j,i})]_{i,j=1}^m$ are Hermitian (symmetric), 
row diagonally dominant with non-negative diagonal entries, this yields
\[ [\Phi (E_{j,i})]_{i,j=1}^m \ge 0 \text{~~and~~}   
[\Psi (E_{j,i})]_{i,j=1}^m \ge 0. \]
So we  complete the proof. 
\end{proof}

We remark that $\Phi (X)=(\mathrm{tr}X)I+ X$ is apparently a completely positive map, 
therefore $\Phi$ is a completely PPT map.  
However, the map  $\Psi (X)= (\mathrm{tr}X)I-X$ is not completely positive 
since it is even not $2$-positive (see \cite{Choi72}),  
thus  $\Psi$ is not a completely PPT map. 

The paper is organized as follows. In Section \ref{sec2}, 
we show a partial traces inequality about PPT matrices 
based on the application of  Theorem \ref{thm1}. 
Some other recent results are implicitly included in our proof of Theorem \ref{thm2}. 
In Section \ref{sec3}, we provide a proof  of  a trace inequality 
that has been applied to quantum information, 
such as, the subadditivity of $q$-entropies and the separability of mixed states.
In the last of the third section, we give some unified extensions of some traces inequalities 
(Theorem \ref{thm8} and Theorem \ref{thm9} ).

\section{Inequalities related to Partial traces}
\label{sec2}

By the completely copositivity of $\Psi$ in  Theorem \ref{thm1}, 
we  get the following Theorem \ref{coro4}, 
which is the main result in \cite[Theorem 2]{Choi17}. 
Of course, the proof provided by  Choi is quite different and technical.

\begin{theorem} \label{coro4}
(see \cite[Theorem 2]{Choi17})
Let $A=[A_{i,j}]_{i,j=1}^m \in \mathbb{M}_m(\mathbb{M}_n)$  be positive semidefinite. Then 
\begin{align*}
 (\mathrm{tr}_2 A^{\tau})\otimes I_n \ge A^{\tau}, \\
I_m \otimes \mathrm{tr}_1 A^{\tau} \ge A^{\tau}.  
\end{align*}
\end{theorem}

\begin{proof}
By Theorem \ref{thm1}, $\Psi (X)=(\mathrm{tr}X)I-X$ is completely copositive, 
that is 
\[ [(\mathrm{tr}A_{j,i})I-A_{j,i} ]_{i,j=1}^m \ge 0.  \]
Under the above definition, we can write 
\begin{equation} \label{eq3}
 (\mathrm{tr}_2 A^{\tau})\otimes I_n \ge A^{\tau}. 
\end{equation}
We may assume that $A_{i.j}=\bigl[a_{r,s}^{i,j}\bigr]_{r,s=1}^n$, then we define 
$ \widetilde{A} \in \mathbb{M}_n (\mathbb{M}_m) $ by 
\[ \widetilde{A}= [B_{r,s}]_{r,s=1}^n, \]
where $B_{r,s}=\bigl[a_{r,s}^{i,j}\bigr]_{i,j=1}^m\in \mathbb{M}_m$. 
By a direct computation, we get 
\[ \mathrm{tr}_2 \widetilde{A} =
\left[ \mathrm{tr} \bigl[a_{r,s}^{i,j}\bigr]_{i,j=1}^m \right]_{r,s=1}^n =
\Bigl[ \begin{matrix} \sum\limits_{i=1}^m  a_{r,s}^{i,i}\end{matrix}  \Bigr]_{r,s=1}^n=
\sum\limits_{i=1}^m \left[ a_{r,s}^{i,i} \right]_{r,s=1}^n=\mathrm{tr}_1A,  \]
and for any $X=[x_{ij}]\in \mathbb{M}_m$ and $Y=[y_{rs}]\in \mathbb{M}_n$, since 
\[ X\otimes Y=[x_{ij}Y]_{i,j=1}^m =\left[ [x_{ij}y_{rs}]_{r,s=1}^n \right]_{i,j=1}^m. \]
Then, it follows that  
\[ \widetilde{X\otimes Y}=\left[ [x_{ij}y_{rs}]_{i,j=1}^m \right]_{r,s=1}^n
=\left[ y_{rs}X \right]_{r,s=1}^n=Y\otimes X.  \]
Replacing $A$ with $\widetilde{A}$ in (\ref{eq3}), we get 
$ (\mathrm{tr}_2 \widetilde{A}^{\tau} )\otimes I_m \ge \widetilde{A}^{\tau}$, 
that is 
\begin{equation} \label{eq5}
 I_m \otimes \mathrm{tr}_1 A^{\tau} = I_m \otimes \mathrm{tr}_2 \widetilde{A^{\tau}} =
\widetilde{(\mathrm{tr}_2 \widetilde{A}^{\tau} )\otimes I_m} \ge 
\widetilde{\widetilde{A^{\tau}}\,\,}=A^{\tau},  
\end{equation}
in which we frequently use the fact that $ \widetilde{A}^{\tau} = \widetilde{A^{\tau}}$. 
\end{proof}

As a byproduct of our proof, we have the following Corollary \ref{coro3}, 
see \cite{Horo96} for details and references to the physics literature.  

\begin{corollary} \label{coro3}
(see \cite{Horo96}) 
Let  $A=[A_{i,j}]_{i,j=1}^m \in \mathbb{M}_m(\mathbb{M}_n)$  be PPT. Then 
\begin{align*}
 I_m\otimes (\mathrm{tr}_1A) \ge A, \\
(\mathrm{tr}_2A) \otimes I_n \ge A. 
\end{align*}
\end{corollary} 

\begin{proof}
Since $A$ and $A^{\tau}$ are positive semidefinite, 
by replacing $A$ with $A^{\tau}$ in Theorem \ref{coro4}, we get the desired results. 
\end{proof}

By combining Theorem \ref{coro4} and Corollary \ref{coro3},
we immediately obtain  the following partial traces inequality. 

\begin{corollary} \label{thm2}
Let $A=[A_{i,j}]_{i,j=1}^m \in \mathbb{M}_m(\mathbb{M}_n)$  be PPT. Then 
\[ I_m \otimes (\mathrm{tr}_1A) + (\mathrm{tr}_2A) \otimes I_n \ge 2A, \]
and 
\[ I_m \otimes (\mathrm{tr}_1A^{\tau}) + (\mathrm{tr}_2A^{\tau}) \otimes I_n \ge 2A^{\tau}. \]
\end{corollary}

\begin{proposition}
Let $A \in \mathbb{M}_2(\mathbb{M}_n)$ 
be positive semidefinite. Then 
\[ I_2\otimes (\mathrm{tr}_1 A) +(\mathrm{tr}_2 A)\otimes I_n \le A + (\mathrm{tr} A)I_{2n}.  \]
\end{proposition}

\begin{proof}
We may assume that 
\[ A=\begin{bmatrix}B & C \\ C^* &D  \end{bmatrix},\] 
where $B,C,D\in \mathbb{M}_n$. The desired inequality is 
\begin{equation*} \begin{aligned} 
& \begin{bmatrix} B+D & 0 \\ 0 & B+D \end{bmatrix} +
\begin{bmatrix} (\mathrm{tr} B)I_n & (\mathrm{tr}C)I_n \\ 
(\mathrm{tr}C^*)I_n & (\mathrm{tr} D)I_n  \end{bmatrix} \\
&\quad \le  \begin{bmatrix}B & C \\ C^* &D  \end{bmatrix} + 
\begin{bmatrix}(\mathrm{tr} A)I_n & 0 \\ 0 & (\mathrm{tr} A)I_n  \end{bmatrix},
\end{aligned} \end{equation*}
which is equivalent to (note that $\mathrm{tr} A= \mathrm{tr}B +\mathrm{tr}D$)
\[ G:= \begin{bmatrix} (\mathrm{tr} D)I_n-D & C- (\mathrm{tr} C)I_n \\
 C^*-(\mathrm{tr} C^*)I_n & (\mathrm{tr} B)I_n- B  \end{bmatrix}\ge 0. \]
By Theorem \ref{thm1}, the map $\Psi (X)=(\mathrm{tr} X)I-X$ is completely copositive, 
\[ \begin{bmatrix} (\mathrm{tr} B)I_n- B & (\mathrm{tr} C^*)I_n-C^* \\
 (\mathrm{tr} C)I_n-C & (\mathrm{tr} D)I_n-D \end{bmatrix} \ge 0, \]
and then 
\[ G= \begin{bmatrix}0& -I_n \\ I_n &0   \end{bmatrix}
\begin{bmatrix} (\mathrm{tr} B)I_n- B & (\mathrm{tr} C^*)I_n-C^* \\
 (\mathrm{tr} C)I_n-C & (\mathrm{tr} D)I_n-D \end{bmatrix} 
\begin{bmatrix}0& I_n \\ -I_n &0   \end{bmatrix} \ge 0.\]
Hence we complete the proof. 
\end{proof}

We remark that if 
$A=[A_{i,j}]_{i,j=1}^m \in \mathbb{M}_m(\mathbb{M}_n)$  is positive semidefinite, 
by induction and the $2$-copositivity of $\Psi (X)=(\mathrm{tr} X)I-X$, one can show that 
\[  I_m \otimes (\mathrm{tr}_1 A ) + (\mathrm{tr}_2 A)\otimes I_n \le A + 
(\mathrm{tr} A) I_{m}\otimes I_n, \]
which was proved   by Ando \cite{Ando14} and independently by Lin \cite{Lin16}. 

%%%%%%%%%%%%
%%%%%%%%%%%%

\section{Inequalities relating to trace}  
\label{sec3}

Recently, Choi established the following partial trace inequalities [Corollary \ref{coro6}], 
which is the key result in  \cite[Theorem 2]{Choi17b}.   
Here we shall demonstrate that Corollary \ref{coro6} is actually a well application of 
the completely copositive $\Phi (X)= (\mathrm{tr} X)I+X$. 
In the sequel, 
we first give an alternatgive proof of Corollary \ref{coro6} based on Theorem \ref{thm1}. 
The Corollary \ref{coro7} can be found in \cite[Theorem 2.2]{Lin17}, 
%gbut the authors leave the details  for the interested reader. 
here we  provide the proof for completeness 
using the completely copositivity of $\Psi (X)=(\mathrm{tr} X)I-X$. 
Some interesting consequences about trace are included. 

\begin{corollary} \label{coro6}
(see \cite{Choi17b})
Let $A \in \mathbb{M}_m(\mathbb{M}_n)$  be positive semidefinite. Then 
\begin{align*} 
(\mathrm{tr}_2 A^{\tau})\otimes I_n \ge -A^{\tau}, \\
I_m \otimes \mathrm{tr}_1 A^{\tau} \ge -A^{\tau}.  
\end{align*}
\end{corollary}

\begin{proof}
In view of symmetry of definitions of $\mathrm{tr}_1$ and $\mathrm{tr}_2$, 
we  only prove 
\[ (\mathrm{tr}_2 A^{\tau})\otimes I_n \ge -A^{\tau}. \] 
By Theorem \ref{thm1}, the map $\Phi (X)=(\mathrm{tr} X)I+X$ is completely copositive, 
then 
\[ [(\mathrm{tr} A_{j,i})I_n +A_{j,i}]_{i,j=1}^m \ge 0, \] 
which can be rewrite as $(\mathrm{tr}_2 A^{\tau})\otimes I_n \ge -A^{\tau}$. 
\end{proof}

\begin{corollary} \label{coro7}
(see \cite{Lin17})
Let $\begin{bmatrix}A & B \\ B^* &C  \end{bmatrix}\in \mathbb{M}_2(\mathbb{M}_n)$ 
be positive semidefinite. Then 
\begin{equation} \label{eq7}
 \begin{bmatrix}
(\mathrm{tr} C)A-BB^* & (\mathrm{tr} B^*)B-AC \\
(\mathrm{tr} B)B^*-CA & (\mathrm{tr} A)C -B^*B
\end{bmatrix} \ge 0. 
\end{equation}
Consequently, 
\begin{gather}  
\label{eq8}  
 \mathrm{tr}(AC) +\mathrm{tr}(B^*B) 
\le \mathrm{tr} A\mathrm{tr} C +|\mathrm{tr} B|^2, \\
\label{eq9}  
 \mathrm{tr}(B^*B) - \mathrm{tr}(AC) 
\le \mathrm{tr} A\mathrm{tr} C -|\mathrm{tr} B|^2.
\end{gather}

\end{corollary}

\begin{proof}
Since $\begin{bmatrix}Y^*Y & Y^*X \\ X^*Y & X^*X  \end{bmatrix}$ is positive semidefinite for 
any $p\times q$ matrices $X,Y$,   
by Theorem \ref{thm1}, the completely copositivity of $\Psi (X)=(\mathrm{tr} X)I-X$ yields 
\[ \begin{bmatrix} (\mathrm{tr} Y^*Y)I -Y^*Y &  ( \mathrm{tr} X^*Y )I-X^*Y \\ 
( \mathrm{tr} Y^*X )I-Y^*X  & ( \mathrm{tr} X^*X)I-X^*X  \end{bmatrix}\ge 0. \]
Now since $\begin{bmatrix}A & B \\ B^* &C  \end{bmatrix}$ is positive semidefinite, 
we may write 
\[ \begin{bmatrix}A & B \\ B^* &C  \end{bmatrix} = 
\begin{bmatrix}XX^* & XY^* \\ YX^* & YY^*  \end{bmatrix}, \]
for some $X,Y\in \mathbb{M}_{n\times 2n}$. We observe that 
\begin{align*}
 & \begin{bmatrix}
(\mathrm{tr} C)A-BB^* & (\mathrm{tr} B^*)B-AC \\
(\mathrm{tr} B)B^*-CA & (\mathrm{tr} A)C -B^*B
\end{bmatrix}  \\
&\quad =  \begin{bmatrix}
(\mathrm{tr} YY^*)XX^*-XY^*YX^* & (\mathrm{tr} YX^*) XY^*-XX^*YY^* \\
(\mathrm{tr} XY^*) YX^*- YY^*XX^* & (\mathrm{tr} XX^*)YY^* - YX^*XY^*
\end{bmatrix} \\
&\quad = 
\begin{bmatrix}X & 0 \\ 0 &Y  \end{bmatrix} 
\begin{bmatrix}
(\mathrm{tr} Y^*Y)I-Y^*Y & (\mathrm{tr} X^*Y) I-X^*Y \\
(\mathrm{tr} Y^*X) I- Y^*X & (\mathrm{tr} X^*X)I - X^*X
\end{bmatrix} 
\begin{bmatrix}X & 0 \\ 0 &Y  \end{bmatrix}^*.  
\end{align*}
Therefore, (\ref{eq7})   follows. Then 
\[ (\mathrm{tr} C)A-BB^* + (\mathrm{tr} A)C -B^*B \ge 
  \pm \bigl( (\mathrm{tr} B^*)B-AC + (\mathrm{tr} B)B^*-CA  \bigr). \]
By taking trace on both sides, it yields (\ref{eq8}) and (\ref{eq9}). 
\end{proof}

Positive semidefinite $2\times 2$ block matrices are well studied, 
such a partition leads to versatile and elegant theoretical inequalities, 
see \cite{Gumus18,Lin14,Lin15,Lin17} for details. 
However, an analogous partition into $3\times 3$ blocks matrices 
seems not to be extensively investigated. 
At the end of the paper, we will 
present several results related to positive semidefinite $3\times 3$ block matrices. 

Let $A$ be an $n\times n$ complex matrix. 
For index sets $\alpha ,\beta \subseteq \{1,2,\ldots ,n\}$, 
we denote by $A[\alpha, \beta]$ the submatrix of entries that lie in the rows of $A$ 
indexed by $\alpha$ and the columns indexed by $\beta$. 
If $\alpha =\beta$, the submatrix $A[\alpha ,\beta]=A[\alpha] $ 
is the principal submatrix of $A$. 
We denoted by $|\alpha|$ the cardinality of the index set $\alpha$.

Recently, Lin and van den Driessche proved 
a determinantal inequality (\ref{eqlin}), 
which is a refinement of the famous Kotelianskii's inequality 
(see, e.g., \cite{Jiang19}),
it states that for any positive semidefinite $A\in \mathbb{M}_n$ 
and $\alpha ,\beta \subseteq \{1,2,\ldots ,n\}$ with $|\alpha |=|\beta |$, then 

\begin{equation} \label{eqlin} \begin{aligned} 
&(\det A[\alpha \cup \beta ]) (\det A[\alpha \cap \beta ]) \\
&\quad \le (\det A[\alpha ])(\det A[ \beta ]) -\bigl| \det A[\alpha ,\beta ] \bigr|^2 .
\end{aligned} \end{equation}

The next two results Theorem \ref{thm8} and Theorem \ref{thm9} 
are extensions of Corollary \ref{coro7}, 
and it also can be regarded as the complement of (\ref{eqlin}).

\begin{theorem} \label{thm8}
Let $A\in \mathbb{M}_n$  be positive semidefinite and 
let $\alpha ,\beta \subseteq \{1,2,\ldots ,n\}$ such that $|\alpha |=|\beta|$. Then 
\begin{equation} \label{eq10} \begin{aligned}
 &  \mathrm{tr} \bigl( A[\alpha]A[\beta]  \bigr)  + 
\mathrm{tr} \bigl( A^*[\alpha, \beta] A[\alpha, \beta] \bigr)   \\
&\quad \le   \bigl( \mathrm{tr} A[\alpha] \bigr) \bigl(\mathrm{tr}A[\beta] \bigr) +  
\bigl| \mathrm{tr}A[\alpha ,\beta] \bigr|^2.
\end{aligned} \end{equation}
\end{theorem} 

\begin{proof}
Without loss of generality, we may assume that 
$\alpha \cup \beta =\{1,2,\ldots ,n\}$ so that $A[\alpha\cup \beta]=A$ 
(otherwise, work within the principal submatrix $A[\alpha, \beta]$). 
By  suitable rearrangement of subscripts or 
by permutational similarity if necessary, we may further assume that 
\begin{equation*}
A=\begin{bmatrix}
A_{1,1} & A_{1,2} &A_{1,3} \\
A_{1,2}^* & A_{2,2} & A_{2,3} \\
A_{1,3}^* &A_{2,3}^* & A_{3,3}
\end{bmatrix},
\end{equation*}
and 
\[ A[\alpha ]= \begin{bmatrix}
A_{1,1} & A_{1,2} \\ A_{1,2}^* & A_{2,2}
\end{bmatrix} ,\quad 
A[\beta ] = \begin{bmatrix}
A_{2,2} & A_{2,3} \\ A_{2,3}^* & A_{3,3}
\end{bmatrix},\quad 
A[\alpha , \beta ] = \begin{bmatrix}
A_{1,2} & A_{1,3} \\ A_{2,2} & A_{2,3}
\end{bmatrix}. \] 
Since $A$ is positive semidefinite, and observe that 
\begin{align*}
 & \begin{bmatrix} 
I & 0 & 0 & 0 \\
0 & I & 0 & 0  \\
0 & -I & I & 0 \\
0 & 0 & 0 & I 
 \end{bmatrix}   
 \begin{bmatrix}
A[\alpha ] & A[\alpha ,\beta] \\
A^*[\alpha,\beta ] & A[\beta ] 
\end{bmatrix} 
\begin{bmatrix} 
I & 0 & 0 & 0 \\
0 & I & -I & 0  \\
0 & 0 & I & 0 \\
0 & 0 & 0 & I 
 \end{bmatrix}  \\
&\quad = 
\begin{bmatrix} 
A_{1,1} & A_{1,2} & 0 & A_{1,3} \\
A_{2,1} & A_{2,2} & 0 & A_{2,3}  \\
0 & 0 & 0 & 0 \\
A_{3,1} & A_{3,2} & 0 & A_{3,3} 
 \end{bmatrix}\ge 0 , 
\end{align*}
therefore, we have  
\[  \begin{bmatrix}
A[\alpha ] & A[\alpha ,\beta] \\
A^*[\alpha,\beta ] & A[\beta ] 
\end{bmatrix} \ge 0. \]
By  (\ref{eq8}) in the previous Corollary \ref{coro7}, 
the desired result (\ref{eq10}) now follows. 
\end{proof}

Using the same idea in the previous proof and combining \cite[Theorem 2.1]{Lin17} 
or (\ref{eq9}), one could also get  the following trace inequality. 

\begin{theorem} \label{thm9}
Let $A\in \mathbb{M}_n$  be positive semidefinite and 
let $\alpha ,\beta \subseteq \{1,2,\ldots ,n\}$ such that $|\alpha |=|\beta|$. Then 
\begin{align*}
 &\left|  \mathrm{tr} \bigl( A[\alpha]A[\beta]  \bigr)- 
\mathrm{tr} \bigl( A^*[\alpha, \beta] A[\alpha, \beta] \bigr)  \right| \\
&\quad \le   \bigl( \mathrm{tr} A[\alpha] \bigr) \bigl( \mathrm{tr}A[\beta] \bigr) -  
\bigl| \mathrm{tr}A[\alpha ,\beta] \bigr|^2.
\end{align*}
\end{theorem}

\section*{Acknowledgments}
The first  author would like to thank Minghua Lin  
for bringing the question to his attention. 
All authors are grateful for valuable comments from the referee, 
which considerably improve the presentation of our manuscript.
This work was supported by  NSFC (Grant No. 11671402, 11871479),  
Hunan Provincial Natural Science Foundation (Grant No. 2016JJ2138, 2018JJ2479) 
and  Mathematics and Interdisciplinary Sciences Project of CSU.

\end{document}